\documentclass[12pt]{article}

\usepackage[]{graphicx,float,latexsym,times}
\usepackage{amsfonts,amstext,amsmath,amssymb,amsthm}
\usepackage{caption2}

\long\def\symbolfootnote[#1]#2{\begingroup%
\def\thefootnote{\fnsymbol{footnote}}\footnote[#1]{#2}\endgroup}

\usepackage{titlesec} 

\titleformat{\section}{\large\bfseries}{\thesection.}{.5em}{}
\titlespacing*{\section}{0pt}{*3}{*2}
\titleformat{\subsection}{\normalfont\bfseries}{\thesubsection.}{.5em}{}
\titlespacing*{\subsection} {0pt}{*3}{*2}
\titleformat{\subsubsection}{\normalfont\bfseries}{\thesubsubsection.}{.5em}{}
\titlespacing*{\subsubsection} {0pt}{*3}{*2}

\newtheorem{theorem}{Theorem}[section]

\newtheorem{lemma}{Lemma}[section]

\newtheorem{algorithm}{Algorithm}[section]

\newtheorem{problem}{Problem}[section]

\newcommand{\Prob}{\mathbb{P}}
\newcommand{\Expect}{\mathbb{E}}
\newcommand{\indic}{\mathbb{I}}

\long\def\symbolfootnote[#1]#2{\begingroup%
\def\thefootnote{\fnsymbol{footnote}}\footnote[#1]{#2}\endgroup}

\DeclareMathOperator*{\esssup}{ess\,sup}

\newcommand{\CADD}{{\mathsf{CADD}}}

\newcommand{\WADD}{{\mathsf{WADD}}}
\newcommand{\FAR}{{\mathsf{FAR}}}

\newcommand{\PDC}{{\mathsf{PDC}}}

\newcommand{\tauc}{\tau_{\scriptscriptstyle \mathrm{C}}}

\newcommand{\taucc}{\tau_{\scriptscriptstyle \mathrm{CC}}}

\newcommand{\tauw}{\tau_{\scriptscriptstyle \mathrm{W}}}
\newcommand{\tauall}{\tau_{\scriptscriptstyle \mathrm{All}}}
\newcommand{\taudeall}{\tau_{\scriptscriptstyle \mathrm{DE-All}}}

\newcommand{\Pideall}{\Pi_{\scriptscriptstyle \mathrm{DE-All}}}
\newcommand{\Piall}{\Pi_{\scriptscriptstyle \mathrm{All}}}

\newcommand{\nuwl}{\nu_{{\scriptscriptstyle \mathrm{W}}, \ell}}
\newcommand{\nucl}{\nu_{{\scriptscriptstyle \mathrm{C}}, \ell}}
\newcommand{\nurl}{\nu_{{\scriptscriptstyle \mathrm{R}}, \ell}}

\newcommand{\tauwl}{\tau_{{\scriptscriptstyle \mathrm{W}}, \ell}}
\newcommand{\taucl}{\tau_{{\scriptscriptstyle \mathrm{C}}, \ell}}
\newcommand{\taurl}{\tau_{{\scriptscriptstyle \mathrm{R}}, \ell}}

\theoremstyle{plain} 

\renewcommand{\refname}{{\large \bf REFERENCES}}
\renewcommand{\captionlabeldelim}{.} 
\renewcommand{\captionlabelfont}{\bfseries\emph} 

\numberwithin{equation}{section} 

\DeclareMathSizes{3}{3}{3}{3}

\usepackage[authoryear]{natbib}

\begin{document}

\title{\textbf{\Large Data-Efficient Minimax Quickest Change Detection in a Decentralized System}}

\date{}

\maketitle
\vspace{-2cm}

\author{
\begin{center}

\textbf{\large Taposh Banerjee and Venugopal. V.\ Veeravalli}

ECE Department and Coordinated Science Laboratory\\
University of Illinois at Urbana-Champaign, USA \\
\end{center}
}

\symbolfootnote[0]{\small Address correspondence to Venugopal V. Veeravalli,
ECE Department, University of Illinois at Urbana-Champaign, 106 Coordinated Science Laboratory,
1308 West Main Street, Urbana, IL 61801, USA; Fax: 217-244-1642 ; E-mail: vvv@illinois.edu.}

{\small \noindent\textbf{Abstract:} A sensor network is considered where a sequence
of random variables is observed at each sensor. At each time
step, a processed version of the observations is transmitted from
the sensors to a common node called the fusion center. At some
unknown point in time the distribution of the observations at all
the sensor nodes changes. The objective is to detect this change
in distribution as quickly as possible, subject to constraints on
the false alarm rate and the cost of observations taken at each
sensor. Minimax problem formulations are proposed for the
above problem. A data-efficient algorithm is proposed in which an adaptive sampling strategy is
used at each sensor to control the cost of observations used before change.
To conserve the cost of communication an occasional binary digit is transmitted 
from each sensor to the fusion center.
It is shown
that the proposed algorithm is globally asymptotically optimal for the proposed formulations, as
the false alarm rate goes to zero.

\section{INTRODUCTION} \label{sec:Intro}

In many engineering applications, e.g., surveillance/monitoring
of infrastructure (bridges, historic buildings, etc) or animal/bird habitat
using sensor networks, there is a need to detect a sudden onset of an unusual activity
or abnormal behavior as quickly as possible.
Such an inference task is often performed by employing a sensor network.
In a sensor network multiple geographically distributed sensors are deployed
to observe a phenomenon. The objective is to detect the onset of the activity/behavior
using the sensor nodes in a collaborative fashion.

In this paper, we study the above detection problem in the framework of decentralized
quickest change detection (QCD) introduced in \cite{veer-ieeetit-2001} and further studied in
\cite{mei-ieeetit-2005} and \cite{tart-veer-sqa-2008}. In the model studied in
these papers, the
observations at the sensors are modeled as random variables, and at each time step
a processed version of the observations is transmitted from the sensors to a
common decision node, called the fusion center. At some point in time, called the change point,
the distribution of the random variables observed at \textit{all} the sensors changes. The objective
is to find a stopping time on the information received at the fusion center, so as to
detect the change in distribution as quickly as possible (with minimum possible delay),
subject to a constraint on the false alarm rate. The observations
are independent across the sensors, and independent and identically distributed 
before and after the change point, conditioned on the change
point. The pre- and post-change distributions are assumed to be known.

In many applications of QCD, including those mentioned above,
changes are rare and acquiring data or taking
observations is costly, e.g., the cost of batteries in sensor networks or
the cost of communication between the sensors
and the fusion center. In \cite{veer-ieeetit-2001}, \cite{mei-ieeetit-2005} and \cite{tart-veer-sqa-2008},
the cost of communication is controlled by quantizing or censoring observations/statistic at the
sensors. However, the cost of taking observations at the sensors is not taken into account.
Motivated by this, we study quickest change detection in sensor networks with an additional
constraint on the cost of observations used at each sensor.

One way to detect a change in the sensor network model discussed above is to use the
\textit{Centralized CuSum} algorithm (to be defined below; also see \cite{page-biometrica-1954}).
In this algorithm, all the observations are taken at each sensor,
and raw observations are transmitted from each sensor to the fusion center. At the fusion center
the \textit{CuSum algorithm} (see \cite{page-biometrica-1954}) is applied to all the received observations.
The Centralized CuSum algorithm is globally asymptotically
optimal since the problem is simply that of
detecting a change in a vector sequence of observations, and hence reduces to the classical QCD problem
studied in \cite{lord-amstat-1971}, \cite{poll-astat-1985}, and \cite{lai-ieeetit-1998}.
The problem is more interesting when sending raw observations from the sensors to the fusion center
is not permitted, and at each sensor quantization of observations is enforced.
A major result in this case is due to \cite{mei-ieeetit-2005},  
in which a CuSum is applied locally at each sensor. A ``1'' is transmitted from
a sensor to a fusion center each time the local CuSum statistic is above a local threshold. A change
is declared at the fusion center when a ``1'' is received from \textit{all} the sensors at the same time.
It is shown in \cite{mei-ieeetit-2005} that the delay of this ALL scheme is asymptotically of the same order as the
delay of the Centralized CuSum scheme (for the same false alarm rate constraint, the ratio of their delay goes to $1$),
as the false alarm rate goes to zero.

However, in applications where changes are rare, the Centralized CuSum algorithm and the ALL algorithm are not
energy efficient. This is because at the sensors, all of the observations are used for decision making,
potentially consuming all the available energy at a sensor before the change even occurs.
In modern sensor networks the most effective way to save energy at a sensor is to switch the sensors
between on and off states, essentially controlling the \textit{duty cycle}. However, introducing
on-off observation control at sensors in an ad-hoc fashion can adversely affect the detection delay.

In \cite{bane-veer-sqa-2012} and \cite{bane-veer-IT-2013} we proposed two-threshold extensions of the classical algorithms
of \cite{shir-siamtpa-1963} and \cite{page-biometrica-1954}, respectively. In these two-threshold algorithms,
an adaptive sampling strategy is used for on-off observation control in their classical counterparts. We showed that
such an observation control can be applied without affecting the first-order asymptotic performance of the classical algorithms.
In this paper we extend the ideas from \cite{bane-veer-IT-2013} by introducing such an adaptive sampling strategy locally at each sensor.

Specifically, in this paper we introduce observation control in the ALL scheme of \cite{mei-ieeetit-2005} by replacing
the CuSum algorithm at each sensor by the DE-CuSum algorithm we proposed in \cite{bane-veer-IT-2013}. We call this new algorithm the DE-All algorithm.
We propose extensions of the data-efficient formulations from \cite{bane-veer-isit-2014}
to sensor networks, and show that the DE-All scheme is globally asymptotically optimal for these formulations.
By global asymptotic optimality we mean that the ratio of the delay the DE-All scheme and
the Centralized CuSum scheme goes to $1$ as the false alarm rate goes to zero.
Thus, one can skip an arbitrary but fixed fraction of samples before change at the sensors, and
transmit just an occasional ``1'' from the sensors to the fusion center, thus conserving significantly the cost of battery,
and yet perform as well (asymptotically up to first order) as the Centralized CuSum algorithm.

\section{Problem Formulation}
\label{sec:ProblemFormulation}

The sensor network is assumed to consist of $L$ sensors and a central decision
maker called the fusion center.
The sensors are indexed by the index $\ell \in \{1, \cdots, L\}$. In the following
we say sensor $\ell$ to refer to the sensor indexed by $\ell$.
At sensor $\ell$ the sequence $\{X_{n,\ell}\}_{n\geq 1}$ is observed, where $n$ is
the time index.
At some unknown time $\gamma$, the distribution of $\{X_{n,\ell}\}$ changes
from $f_{0,\ell}$ to say $f_{1,\ell}$, $\forall \ell$.
The random variables $\{X_{n,\ell}\}$ are independent across indices $n$ and $\ell$
conditioned on $\gamma$.
The distributions $f_{0,\ell}$ and $f_{1,\ell}$ are assumed to be known.

We now discuss the type of policies considered in this paper.
In the quickest change detection models studied in
\cite{veer-ieeetit-2001}, \cite{mei-ieeetit-2005} and \cite{tart-veer-sqa-2008},
observations are taken at each sensor at all times. Here
we consider policies in which on-off observation control is employed at each sensor.
At sensor $\ell$, at each time $n, n\geq 0$,
a decision is made as to whether to \textit{take} or \textit{skip} the observation at time $n+1$ at that sensor.
Let $S_{n,\ell}$ be the indicator random variable such that
\[
S_{n,\ell} = \begin{cases}
1 & \text{~if~} X_{n,\ell} \text{ is used for decision making at sensor $\ell$}\\
0 & \text{ otherwise }.
\end{cases}
\]

Let $\phi_{n, \ell}$ be the observation control law at sensor $\ell$, i.e.,
\begin{equation*}
S_{n+1,\ell} = \phi_{n, \ell}(\mathcal{I}_{n,\ell}),
\end{equation*}
where
$\mathcal{I}_{n,\ell} = \left[ S_{1,\ell}, \ldots, S_{n,\ell}, X_{1,\ell}^{(S_{1,\ell})}, \ldots, X_{n,\ell}^{(S_{n,\ell})} \right]$.
Here, $X_{n,\ell}^{(S_{n,\ell})} = X_{1,\ell}$ if $S_{1,\ell}=1$, otherwise $X_{1,\ell}$ is absent from the
information vector $\mathcal{I}_{n,\ell} $.
Thus, the decision to take or skip a sample at sensor $\ell$ is based on its past
information.
Let
\begin{equation*}
Y_{n,\ell} = g_{n,\ell}(\mathcal{I}_{n,\ell})
\end{equation*}
be the information transmitted from sensor $\ell$ to the fusion center. If no information
is transmitted to the fusion center, then $Y_{n,\ell}=\text{NULL}$, which is treated as zero
at the fusion center. Here,
$g_{n,\ell}$ is the transmission control law at sensor $\ell$.
Let
\[\boldsymbol{Y}_n = \{Y_{n,1}, \cdots, Y_{n,L}\}\]
be the information received at the fusion center at time $n$,
and let $\tau$ be a
stopping time on the sequence $\{\boldsymbol{Y}_n\}$.

Let
\[\boldsymbol{\phi_n}=\{\phi_{n, 1}, \cdots, \phi_{n, L}\}\]
denote the observation control law at time $n$, and let
\[\boldsymbol{g}_n=\{g_{n, 1}, \cdots, g_{n, L}\}\]
denote the transmission control law at time $n$.
For data-efficient quickest change detection in sensor networks
we consider the policy of type $\Pi$
defined as
\[\Pi = \{\tau, \{\boldsymbol{\phi}_0, \cdots, \boldsymbol{\phi}_{\tau-1}\}, \{\boldsymbol{g}_1, \cdots, \boldsymbol{g}_\tau\}\}.\]

To capture the cost of observations used at each sensor before change,
we use the following Pre-Change Duty Cycle ($\PDC$) metric proposed in \cite{bane-veer-isit-2014}.
The $\PDC_\ell$, the $\PDC$ for sensor $\ell$, is defined as
\begin{equation}
\label{eq:PDC_l}
\PDC_\ell(\Pi) = \limsup_{\gamma \to \infty}  \frac{1}{\gamma} \Expect_\infty \left[\sum_{k=1}^{\gamma-1} S_{k,\ell}\right].
\end{equation}
Thus, $\PDC_\ell$ is the fraction of time observations are taken before change at sensor $\ell$.
If all the observations are used at sensor $\ell$, then $\PDC_\ell =1$.
If every second sample is skipped at sensor $\ell$, then $\PDC_\ell=0.5$.

We now propose data-efficient extensions of problems from \cite{bane-veer-isit-2014}
for sensor networks. In \cite{bane-veer-isit-2014} we considered extensions of two popular
minimax formulations: one due to \cite{lord-amstat-1971} and another due to \cite{poll-astat-1985}.
Let
\[
\boldsymbol{\mathcal{I}}_n=\{\mathcal{I}_{n, 1}, \cdots, \mathcal{I}_{n, L}\}
\]
be the information available at time $n$ across the sensor network.
We first consider the delay and false alarm metrics used in \cite{lord-amstat-1971}:
the Worst case Average Detection Delay ($\WADD$)
\begin{equation}\label{eq:WADD_DistSys_Def}
\WADD(\Pi) = \sup_{\gamma \geq 1}  \esssup\; \Expect_\gamma \left[ (\tau-\gamma)^+ | \boldsymbol{\mathcal{I}}_{\gamma-1} \right].
\end{equation}
and the False Alarm Rate ($\FAR$)
\begin{equation}
\FAR(\Pi) = 1/\Expect_\infty \left[ \tau\right].
\end{equation}
The first data-efficient formulation for sensor network that we consider in this paper is
\begin{problem}\label{prob:DE-CENSORLorden}
\begin{eqnarray}
\underset{\Pi}{\text{minimize}} && \WADD(\Pi),\nonumber \\
\text{subject to } && \FAR(\Pi) \leq \alpha, \\
\text{           } &&\PDC_\ell(\Pi) \leq \beta_\ell, \; \mbox{ for } \; \ell=1,\cdots,L. \nonumber
\end{eqnarray}
Here, $0 \leq \alpha, \beta_\ell \leq 1$, for $\ell=1,\cdots,L$, are given constraints.
\end{problem}

We also consider the formulation where instead of $\WADD$, the $\CADD$ metric
\begin{equation}\label{eq:CADD_Sensor_Def}
\CADD(\Pi) = \sup_\gamma \ \ \Expect_\gamma \left[ \tau-\gamma | \tau \geq \gamma \right]
\end{equation}
is used:
\begin{problem}\label{prob:DE-CENSORPollak}
\begin{eqnarray}
\underset{\Pi}{\text{minimize}} && \CADD(\Pi),\nonumber \\
\text{subject to } && \FAR(\Pi) \leq \alpha, \\
\text{           } &&\PDC_\ell(\Pi) \leq \beta_\ell, \; \mbox{ for } \; \ell=1,\cdots,L. \nonumber
\end{eqnarray}
Here, $0 \leq \alpha, \beta_\ell \leq 1$, for $\ell=1,\cdots,L$, are given constraints.
\end{problem}

In \cite{lai-ieeetit-1998} an asymptotic lower bound on the $\CADD$ of any stopping rule satisfying
an $\FAR$ constraint of $\alpha$ is obtained. This bound when specialized to sensor networks is provided
below.
Let
\[\Delta_\alpha = \{\Pi: \FAR(\Pi) \leq \alpha\}.\]
\begin{theorem}[\cite{lai-ieeetit-1998}]\label{thm:Sensor_LB}
As $\alpha \to 0$,
\begin{equation}\label{eq:LBCADD}
\inf_{\Pi \in \Delta_\alpha} \CADD(\Pi) \geq \frac{|\log \alpha|}{\sum_{\ell=1}^L D(f_{1,\ell} \; || \; f_{0,\ell})} (1+o(1)).
\end{equation}
\end{theorem}
\medskip
Since $\WADD(\Pi) \geq \CADD(\Pi)$, we also have as $\alpha \to 0$,
\begin{equation}\label{eq:LBWADD}
\inf_{\Pi \in \Delta_\alpha} \WADD(\Pi) \geq \frac{|\log \alpha|}{\sum_{\ell=1}^L D(f_{1,\ell} \; || \; f_{0,\ell})} (1+o(1)).
\end{equation}
We note that the lower bound on the $\WADD$ was first obtained in \cite{lord-amstat-1971}.

It is well known that the Centralized CuSum algorithm achieves the lower bound in Theorem~\ref{thm:Sensor_LB} as $\alpha \to 0$; see \cite{lai-ieeetit-1998}.
See the next section for a precise statement. However, in this algorithm raw observations are transmitted from the sensors
to the fusion center at all times. A more communication efficient scheme is the ALL scheme of \cite{mei-ieeetit-2005}, where
information an occasional ``1'' is transmitted from the sensors to the fusion center.
As shown in \cite{mei-ieeetit-2005}, the ALL scheme also achieves the lower bound provided in Theorem~\ref{thm:Sensor_LB}, as $\alpha \to 0$.
However, in the ALL scheme, observations are used at each sensor all the time. We are interested in schemes
that are also data-efficient locally at each sensor.

We will be particularly interested in policies such that the information transmitted from the
sensors to the fusion center at any time is a binary digit.
That is we are primarily interested in policies in the class
\begin{equation}\label{def:BinaryQuantization}
\Delta_{(\alpha,\beta)}^{\{0,1\}} = \{\Pi: \FAR(\Pi) \leq \alpha; \; \PDC_\ell \leq \beta_\ell \; \mbox{ and }\; Y_{n,\ell} \in \{0, 1\}, \; \forall n,\ell\}.
\end{equation}
The interest in the above policies, where only a binary number is sent to the fusion center, stems from
the fact that in these policies the information transmitted to the fusion center is the minimal.
Thus, it represents in some sense the maximum possible compression of the transmitted information.
The main objective of this paper is to show that an algorithm from this class can be globally
asymptotically optimal.

Specifically, we will propose an algorithm, called the DE-All algorithm, from the class $\Delta_{(\alpha,\beta)}^{\{0,1\}}$,
and show that it is asymptotically optimal for both Problem~\ref{prob:DE-CENSORLorden} and Problem~\ref{prob:DE-CENSORPollak},
i.e., the performance of the DE-All algorithm achieves the lower bound of \cite{lai-ieeetit-1998} given in Theorem~\ref{thm:Sensor_LB} above,
for each fixed set of $\{\beta_\ell\}$, as $\alpha \to 0$.

\section{Quickest Change Detection in Sensor Networks: Existing Literature}
In this section we provide a brief overview of the existing literature relevant to this chapter.

We first describe the Centralized CuSum algorithm in a mathematically precise way.
\begin{algorithm}[Centralized CuSum]\label{algo:CentCuSum}
Fix a threshold $A\geq 0$.
\begin{enumerate}
\item Use all the observations at the sensors, i.e.,
\[
S_{n,\ell}=1, \; \forall n,\ell.
\]
\item Raw observations are transmitted from the sensors
to the fusion center at each time step, i.e.,
\[
Y_{n,\ell}=X_{n,\ell}\; \forall n,\ell.
\]
\item The CuSum algorithm (see \cite{page-biometrica-1954}) is applied to the
vector of observations received at the fusion center.
That is, at the fusion center, the sequence $\{V_n\}$  is computed
according to the following recursion: $V_0=0$, and for $n\geq 0$,
\begin{equation}
\label{eq:CentCUSUM}
V_{n+1} = \max \left\{0, \; V_n + \sum_{\ell=1}^L \log \frac{f_{1,\ell}(X_{n+1,\ell})}{f_{0,\ell}(X_{n+1,\ell})}\right\}.
\end{equation}
A change is declared
the first time $V_n$ is above a threshold $A>0$:
\[
\taucc = \inf\left\{ n\geq 1: V_n > A \right\}.
\]
\end{enumerate}
\end{algorithm}

It is well known from \cite{lord-amstat-1971} that the performance of the
Centralized CuSum algorithm is asymptotically equal to the
lower bound provided in Theorem~\ref{thm:Sensor_LB}.
\begin{theorem}[\cite{lord-amstat-1971}]\label{thm:CentCuSumOpt}
If $\Expect_1[\log f_{1,\ell}(X_{1,\ell})/f_{0,\ell}(X_{1,\ell})]$ is finite and positive for each $\ell$.
Then with $A = |\log \alpha|$, we have
\begin{equation}\label{eqn:CentCuSum_Opt}
\begin{split}
\FAR(\taucc) &\leq \alpha, \\
\WADD(\taucc) &\leq \frac{|\log \alpha|}{\sum_{\ell=1}^L D(f_{1,\ell} \; || \; f_{0,\ell})} (1+o(1)) \mbox{ as } \alpha \to 0.
\end{split}
\end{equation}
\end{theorem}

We now describe the ALL algorithm from \cite{mei-ieeetit-2005}.
Let \[d_\ell = \frac{D(f_{1,\ell} \; || \; f_{0,\ell})}{\sum_{k=1}^L D(f_{1,k} \; || \; f_{0,k})}.\]
\begin{algorithm}[ALL]\label{algo:ALL}
Start with $C_{0,\ell}=0$, $\forall \ell$, and fix $A\geq 0$.
\begin{enumerate}
\item At each sensor $\ell$ the CuSum statistic is computed over time:
\[
C_{n+1,\ell} = \max \left\{0, \; C_{n,\ell} + \log \frac{f_{1,\ell}(X_{n+1,\ell})}{f_{0,\ell}(X_{n+1,\ell})}\right\}.
\]
Thus $S_{n,\ell}=1$ $\forall n, \ell$.
\item A ``1'' is transmitted from sensor $\ell$ to the fusion center
if the CuSum statistic is above a threshold $d_\ell A$, i.e,
\[
Y_{n,\ell} = \indic_{\{C_{n,\ell} > d_\ell A\}}.
\]
\item A change is declared when a ``1'' is transmitted from all the sensors at the same time, i.e.,
\[
\tauall = \inf\{n \geq 1: Y_{n,\ell} = 1 \; \forall \ell\}.
\]
\end{enumerate}
\end{algorithm}

The ALL algorithm has a surprising optimality property proved in \cite{mei-ieeetit-2005}.
\begin{theorem}[\cite{mei-ieeetit-2005}]\label{thm:AllOpt}
If the absolute moments up to the third order of $\log f_{1,\ell}(X_{1,\ell})/f_{0,\ell}(X_{1,\ell})$ are finite and positive under $\Prob_1$.
Then with $A = |\log \alpha|$, we have as $\alpha \to 0$.
\begin{equation}\label{eqn:ALL_Opt}
\begin{split}
\FAR(\tauall) &\leq \alpha (1+o(1)), \\
\WADD(\tauall) &\leq \frac{|\log \alpha|}{\sum_{\ell=1}^L D(f_{1,\ell} \; || \; f_{0,\ell})} (1+o(1)).
\end{split}
\end{equation}
\end{theorem}
Thus, the ALL scheme achieves the asymptotic lower bound in Theorem~\ref{thm:Sensor_LB}, which is also
the performance of the Centralized CuSum algorithm. In this sense, the ALL scheme is
globally asymptotically optimal as the false alarm rate goes to zero. It is important to note
that such an optimality is obtained by sending such a minimal amount of information (binary digits)
from the sensors to the fusion center.

However, we note that $\PDC_\ell=1$, $\forall \ell$, for both the Centralized CuSum algorithm and the ALL algorithm.
Hence, neither the Centralized CuSum algorithm nor the ALL algorithm are asymptotically optimal for
Problem~\ref{prob:DE-CENSORLorden} and Problem~\ref{prob:DE-CENSORPollak}, when
$\beta_\ell<1$, for any $\ell$.

Consider a policy in which, at each sensor every $n^{th}$ sample is used, and raw observations
are transmitted from each sensor to the fusion center, each time an observation is taken.
At the fusion center, the CuSum algorithm, as defined above, is applied to the received samples.
In this policy, the $\PDC_\ell$ achieved is equal to $1/n$, $\forall \ell$.
Using this scheme, any given constraints on
the $\PDC_\ell$ can be achieved by using every $n^{th}$ sample,
and by choosing a suitably large $n$.
However, the detection delay for this scheme would be approximately $n$ times
that of the delay for the Centralized CuSum algorithm, for the same false
alarm rate.


\section{The $\text{DE-CuSum}$ Algorithm: Adaptive Sampling Based Extension of the $\text{CuSum}$ Algorithm}
In this section we review the DE-CuSum algorithm we proposed in \cite{bane-veer-IT-2013}.

Suppose we have a single sensor system. By a single sensor system we mean $L=1$, i.e.,
there is only a single sensor, and the fusion center has access to all the information available at that sensor.
This is of course equivalent to having no fusion center and assuming that the stopping decision is made at the only sensor in the system.
Since, there is only one sensor, we can suppress the subscript $\ell$. We thus have the following model.

We have a single sequence of random variables $\{X_n\}$. The random variables are i.i.d. with p.d.f. $f_0$ before change,
and are i.i.d. with p.d.f. $f_1$ after the change point $\gamma$. For a single sensor system we consider the following policies.
Let $S_n$ be the indicator random variable such that
\begin{equation*}
\begin{split}
               S_{n} & =  \begin{cases}
               1 & \mbox{ if } X_n \mbox{ used for decision making }\\
               0 & \mbox{ otherwise. }
               \end{cases}
               \end{split}
\end{equation*}
The information available at time $n$ is denote by
\[\mathcal{I}_n = \{X_1^{(S_1)}, \cdots, X_n^{(S_n)}\},\]
where $X_k^{(S_k)}=X_k$ if $S_k=1$, else $X_k$ is absent from $\mathcal{I}_n$,
and
\[S_n = \phi_n (\mathcal{I}_{n-1}).\]
Here, $\phi_n$ denotes the control map.
Let $\tau$ be a stopping time for the sequence $\{\mathcal{I}_n\}$. A control policy is the collection
\[\Psi = \{\tau, \phi_1, \cdots, \phi_{\tau}\}.\]

We define 
\begin{equation}
\label{eq:PDC}
\PDC(\Psi) = \limsup_{\gamma \to \infty}  \frac{1}{\gamma} \Expect_\infty \left[\sum_{k=1}^{\gamma-1} S_{k}\right].
\end{equation}
For a single sensor system, we are interested in the following problems:
\begin{problem}\label{prob:DELorden}
\begin{eqnarray}
\underset{\Psi}{\text{minimize}} && \WADD(\Psi),\nonumber \\
\text{subject to } && \FAR(\Psi) \leq \alpha, \\
\text{           } &&\PDC(\Psi) \leq \beta, \nonumber
\end{eqnarray}
where $0 \leq \alpha, \beta \leq 1$ are given constraints;
\end{problem}
\begin{problem}\label{prob:DEPollak}
\begin{eqnarray}
\underset{\Psi}{\text{minimize}} && \CADD(\Psi),\nonumber \\
\text{subject to } && \FAR(\Psi) \leq \alpha, \\
\text{           } &&\PDC(\Psi) \leq \beta, \nonumber
\end{eqnarray}
where $0 \leq \alpha, \beta \leq 1$ are given constraints.
\end{problem}

Recall that in the CuSum algorithm, the CuSum statistics $\{C_n\}$
evolves according to the following recursion:
$C_0=0$, and for $n\geq 0$,
\begin{equation}
\label{eq:CUSUM}
C_{n+1} = \max\{0, C_n + \log [f_1(X_{n+1})/f_0(X_{n+1})]\}.
\end{equation}
Thus, in the CuSum algorithm,
the log likelihood ratio of the observations are accumulated over time.
If the statistic $C_n$ goes below $0$, it is reset to $0$. A change is declared
the first time $C_n$ is above a threshold $A>0$:
\[\tauc = \inf\left\{ n\geq 1: C_n > A \right\}.\]
We note that $\PDC(\tauc)=1$.

In \cite{bane-veer-IT-2013}, we proposed the DE-CuSum algorithm that can be used
to detect the change in a data-efficient way.
\begin{algorithm}[$\mathrm{DE-CuSum}$]
\label{algo:DECuSum}
Start with $W_0=0$ and fix $\mu>0$, $A>0$ and $h\geq0$. For $n\geq 0$ use the following control:
\begin{enumerate}
\item Take the first observation.
\item If an observation is taken then update the statistic using
\[ W_{n+1} = (W_n + \log [f_1(X_{n+1})/f_0(X_{n+1})])^{h+},\]
where $(x)^{h+} = \max\{x, -h\}$.
\item If $W_n < 0$, skip the next observation and update the statistic using
\[
W_{n+1} = \min\{W_n + \mu, 0\}.
\]
\item Declare change at
\[\tauw = \inf\left\{ n\geq 1: W_n > A \right\}.\]
\end{enumerate}
\end{algorithm}
\begin{figure}[htb]
\center
\includegraphics[width=9cm, height=5cm]{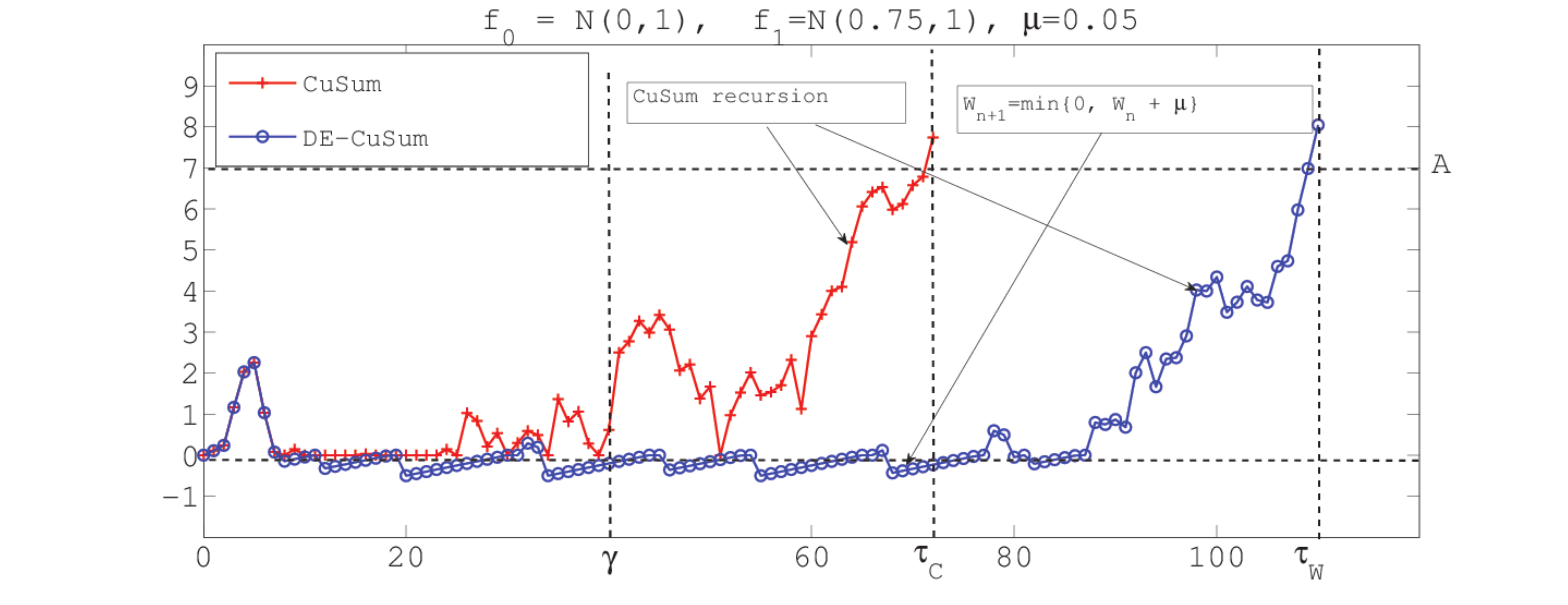}
\caption{{Typical evolution of the CuSum and DE-CuSum statistics evaluated using the same set of samples:
$f_0 \sim {\cal N}(0,1)$, $f_1 \sim {\cal N}(0.75,1)$, $\Gamma=40$, $A=7$,
$\mu=0.05$, and $h=0.5$. Thus, the undershoots are truncated at $-0.5$. Note that $C_n \geq W_n$, $\forall n$.}}
\label{fig:DECuSum_evolution}
\end{figure}

If $h=0$, the DE-CuSum statistic $W_n$ never becomes negative and hence
reduces to the CuSum statistic and evolves as: $W_0=0$, and for $n\geq 0$,
\[W_{n+1} = \max\{0, W_n + \log [f_1(X_{n+1})/f_0(X_{n+1})]\}.\]
Thus, with $h=0$, the DE-CuSum algorithm reduces to the CuSum algorithm.

The evolution of the DE-CuSum algorithm is plotted in Fig.~\ref{fig:DECuSum_evolution}.
If $h=\infty$, the evolution of the DE-CuSum algorithm can be described as follows.
As seen in Fig.~\ref{fig:DECuSum_evolution}, initially the DE-CuSum statistic
evolves according to the CuSum statistic till the statistic $W_n$ goes below $0$.
Once the statistic goes below $0$, samples are skipped depending
on the undershoot of $W_n$ (this is also the sum of the log likelihood of the observations) and a
pre-designed parameter $\mu$.
Specifically, the statistic is incremented by $\mu$ at each time step, and
samples are skipped till $W_n$ goes above zero,
at which time it is reset to zero. At this point, fresh observations are taken and the
process is repeated till the statistic crosses the threshold $A$, at which time a change is declared.
Thus, the DE-CuSum algorithm can also be seen as a sequence of SPRTs (\cite{wald-wolf-amstat-1948},
\cite{sieg-seq-anal-book-1985}) intercepted by ``sleep'' times controlled by the undershoot
and the parameter $\mu$.
If $h < \infty$, the maximum number of consecutive samples skipped is bounded by $h/\mu + 1$.
This may be desirable in some applications.
The following theorem is proved in \cite{bane-veer-IT-2013} and \cite{bane-veer-isit-2014}.
We define the ladder variable (see \cite{wood-nonlin-ren-th-book-1982}, \cite{sieg-seq-anal-book-1985})
\[\tau_{-}= \inf\left\{n \geq 1: \sum_{k=1}^n \log \frac{f_1(X_k)}{f_0(X_k)} < 0\right\}.\]
We note that $W_{\tau_{-}}$ is the ladder height. Recall that $(x)^{h+} = \max\{x, -h\}$.
\begin{theorem}[\cite{bane-veer-IT-2013}, \cite{bane-veer-isit-2014}]
\label{thm:DECuSumOpt}
Let $\Expect_1[\log [f_{1}(X_{1})/f_{0}(X_{1})]]$ and $\Expect_\infty[\log [f_{0}(X_{1})/f_{1}(X_{1})]]$ be finite and positive.
If $\mu > 0$, $h < \infty$, and $A=|\log \alpha|$, we have
\begin{equation}
\label{eq:DECuSumPerf}
\begin{split}
C_n &\geq W_n\; \ \ \ \; \forall n \geq 0,\\
\FAR(\tauw) &\leq \FAR(\tauc) \leq \alpha,\\
\PDC(\tauw) &=\frac{\Expect_\infty[\tau_{-}]}{\Expect_\infty[\tau_{-}] + \Expect_\infty[\lceil | W_{\tau_{-}}^{h+}|/\mu \rceil]},\\
\CADD(\tauw) &\sim \CADD(\tauc) \sim \frac{|\log \alpha|}{D(f_1 \; ||\; f_0)} (1+o(1)) \mbox{ as } \alpha \to 0,\\
\WADD(\tauw) &\sim \WADD(\tauc) \sim \frac{|\log \alpha|}{D(f_1 \; ||\; f_0)} (1+o(1)) \mbox{ as } \alpha \to 0.
\end{split}
\end{equation}
If $h=\infty$, then
\begin{equation}
\label{eq:PDCApprox}
\PDC(\tauw) \leq \frac{\mu}{\mu+D( f_0 \; ||\; f_1)}.
\end{equation}
\end{theorem}
\medskip

Thus, according to the above theorem, the $\FAR$ of the DE-CuSum algorithm is at least as good as that of the CuSum algorithm.
This is because $C_n \geq W_n$, and hence the DE-CuSum statistic will cross the threshold $A$ only after the CuSum statistic has crossed it.
As a result, we can use the same threshold $A$ that we would use in the CuSum algorithm, to satisfy a given constraint on the $\FAR$.
We also note from the expression or the bound on $\PDC$ that it can be made smaller by choosing a smaller $\mu$. So, for any given $\beta$, a suitable
$\mu$ can be chosen to satisfy the $\PDC$ constraint. Also note that $\PDC$ is not a function of the threshold $A$, and hence
the $\PDC$ constraint can be satisfied independent of the constraint on $\FAR$.
Also,
the $\CADD$ and the $\WADD$ of the two algorithms are asymptotically the same (we have actually shown that the delays are within a constant of each other).

These statements together imply that the DE-CuSum algorithm
is asymptotically optimal for Problem~\ref{prob:DELorden} and Problem~\ref{prob:DEPollak}, for each fixed $\beta$, as $\alpha \to 0$.
The fact that the $\PDC$ is not a function of the threshold $A$ is crucial to the proof. In general for a policy $\Psi$, the $\PDC$ can change
as we change the false alarm rate. As a result, as $\alpha \to 0$, the constraint $\beta$ on the $\PDC$ can be violated.
Also, the assumption that $h < \infty$ is an important assumption in the theorem. If $h =\infty$, then the undershoot of the DE-CuSum statistic
can be large. This can cause delay to grow to infinity as we consider the average of worst possible realizations.

To summarize, the DE-CuSum algorithm is a likelihood ratio based adaptive sampling strategy introduced in the CuSum algorithm. 
And in the above theorem it is shown that such a sampling strategy can be added to the CuSum algorithm without any loss in asymptotic performance.
In the next section we use this sampling strategy locally at each sensor.

\section{The $\text{DE-All}$ Algorithm}
\label{sec:DEALL_Algo}
We now propose the main algorithm of this paper, the DE-All algorithm.
In the DE-All algorithm, the DE-CuSum algorithm (see Algorithm~\ref{algo:DECuSum})
is used at each sensor,
 and a ``1'' is transmitted
each time the DE-CuSum statistic at any sensor is above a threshold. A change is declared the first
time a ``1'' is received at the fusion center
 from \textit{all} the sensors at the same time.

We use $W_{n,\ell}$ to denote the DE-CuSum statistic at sensor $\ell$.
Recall that
\[d_\ell = \frac{D(f_{1,\ell} \; || \; f_{0,\ell})}{\sum_{k=1}^L D(f_{1,k} \; || \; f_{0,k})}.\]
\begin{algorithm}[$\mathrm{DE-All}$]
\label{algo:DE-All}
Start with $W_{0,\ell}=0$ $\forall \ell$. Fix $\mu_\ell > 0$, $h_\ell\geq0$, and  $A \geq 0$.
For $n\geq 0$ use the following control:
\begin{enumerate}
\item Use the DE-CuSum algorithm at each sensor $\ell$, i.e., update the statistics $\{W_{n,\ell}\}_{\ell=1}^L$
for $n\geq 1$ using
\begin{equation*}
\begin{split}
S_{n+1, \ell} &= 1 \text{~only if~} W_{n,\ell} \geq 0 \\
W_{n+1,\ell} &= \min\{W_{n,\ell} + \mu_\ell, 0\} \; \text{~if~} S_{n+1,\ell} = 0\\
             &=\left(W_{n,\ell} + \log \frac{f_{1,\ell}(X_{n+1,\ell})}{f_{0,\ell}(X_{n+1,\ell})}\right)^{h+} \; \text{~if~} S_{n+1} = 1\\
\end{split}
\end{equation*}
where $(x)^{h+} = \max\{x, -h\}$.
\item Transmit
\[
Y_{n,\ell} = \indic_{\{W_{n,\ell} > d_\ell A\}}.
\]
\item At the fusion center stop at
\[
\taudeall = \inf\{n \geq 1: Y_{n,\ell}=1 \mbox{ for all } \ell \in \{1, \cdots, L\}\}.
\]
\end{enumerate}
\end{algorithm}
If $h_\ell=0$ $\forall \ell$ then the DE-CuSum algorithm used at each sensor
reduces to the CuSum algorithm. Hence, the $\mathrm{DE-All}$ algorithm reduces to the ALL algorithm;
see Algorithm~\ref{algo:ALL}.

\section{Asymptotic Optimality of the DE-All Algorithm}
In this section we prove the asymptotic optimality of the DE-All algorithm proposed in the previous section.

We define the ladder variable (see \cite{wood-nonlin-ren-th-book-1982}, \cite{sieg-seq-anal-book-1985}) corresponding to sensor $\ell$:
\[
\tau_{\ell-} = \inf\left\{n \geq 1: \sum_{k=1}^n \log \frac{f_{1,\ell}(X_{k,\ell})}{f_{0,\ell}(X_{k,\ell})} < 0\right\}.\]
We note that $W_{\tau_{\ell-}}$ is the ladder height.
\begin{theorem}\label{thm:DEALL}
Let moments of up to third order for the K-L divergences at each sensor be finite and positive.
Let $\mu_\ell > 0$, $h_\ell < \infty$, $\forall \ell$, and $A=|\log \alpha|$. Then we have
\begin{equation}
\label{eq:DEAllPerf}
\begin{split}
\FAR(\Pideall) &\leq \FAR(\Piall) = \alpha(1+o(1)), \mbox{ as } \alpha \to 0,\\
\PDC_\ell(\Pideall) &=\frac{\Expect_\infty[\tau_{\ell-}]}{\Expect_\infty[\tau_{\ell-}] +
\Expect_\infty[\lceil | W_{\tau_{\ell-}}^{h_\ell+}|/\mu_\ell \rceil]},\\
\WADD(\Pideall) &= \frac{|\log \alpha|}{\sum_{\ell=1}^L D(f_{1,\ell} \; || \; f_{0,\ell})}  (1+o(1)) \mbox{ as } \alpha \to 0.
\end{split}
\end{equation}
If $h_\ell=\infty$, $\forall \ell$, then
\begin{equation}
\label{eq:PDCellApprox}
\PDC_\ell(\Pideall) \leq \frac{\mu_\ell}{\mu_\ell+D( f_{0,\ell} \; ||\; f_{1,\ell})} \; \forall \ell.
\end{equation}
\end{theorem}
\begin{proof}
The $\FAR$ result follows from the $\FAR$ result in Theorem~\ref{thm:DECuSumOpt} and the $\FAR$ result of $\Piall$ from
\cite{mei-ieeetit-2005}, because $C_{n,\ell} \geq W_{n,\ell}$, $\forall n, \ell$.
The results on $\PDC_\ell$ follows from renewal reward theorem and is identical to that on $\PDC$ in Theorem~\ref{thm:DECuSumOpt}.
The proof of the $\WADD$ is more involved and is based on the properties of the DE-CuSum algorithm.
The proof is provided in Section~\ref{sec:DEALL_Proof}.
\end{proof}
Since $\CADD \leq \WADD$, we also have under the same assumptions as above
\begin{equation}\label{eq:DEAll_CADDUB}
\begin{split}
\CADD(\Pideall) \leq \frac{|\log \alpha|}{\sum_{\ell=1}^L D(f_{1,\ell} \; || \; f_{0,\ell})}  (1+o(1)) \mbox{ as } \alpha \to 0.
\end{split}
\end{equation}

The above results prove that
the $\mathrm{DE-All}$ algorithm is asymptotically optimal for both Problem~\ref{prob:DE-CENSORLorden}
and Problem~\ref{prob:DE-CENSORPollak}, for each given
$\{\beta_\ell\}$, as $\alpha \to 0$.
This is because the $\WADD$ of $\Pideall$ is asymptotically equal to the lower bound provided in
From Theorem~\ref{thm:Sensor_LB}, as $\alpha \to 0$, and the $\PDC_\ell$ is not a function of threshold $A$. Hence,
the $\PDC_\ell$ constraints can be satisfied
independent of the $\FAR$ constraint $\alpha$.

\section{Numerical Results}
In Fig.~\ref{fig:CompPlot} we compare the $\CADD$ performance as a function of the $\FAR$,
of the ALL scheme, the DE-All algorithm, and the fractional sampling scheme. 
In the fractional sampling scheme, the ALL scheme is used and to meet the constraint on $\PDC_\ell$, 
samples are skipped randomly locally at each sensor. 

The parameters used in the simulations are: $L=10$,
$f_0 = f_{0,\ell}=\mathcal{N}(0,1)$, $\forall \ell$, and $f_1=f_{1,\ell}=\mathcal{N}(0.4,1)$, $\forall \ell$.
The values of $\mu = \mu_\ell = 0.2$, and $h=h_\ell=20$ are used to satisfy a $\PDC_\ell$ constraint of $0.65$ for each $\ell$. 

As shown in the figure the DE-All algorithm provides significant gain as compared to the fractional sampling scheme. 
In general, the gap in performance between the DE-All scheme and the fractional sampling scheme increases as the 
Kullback-Leibler divergence between the pre- and post-change distributions increases.

\label{sec:numerical}
\begin{figure}[htb]
\center
\includegraphics[width=10.5cm, height=6.5cm]{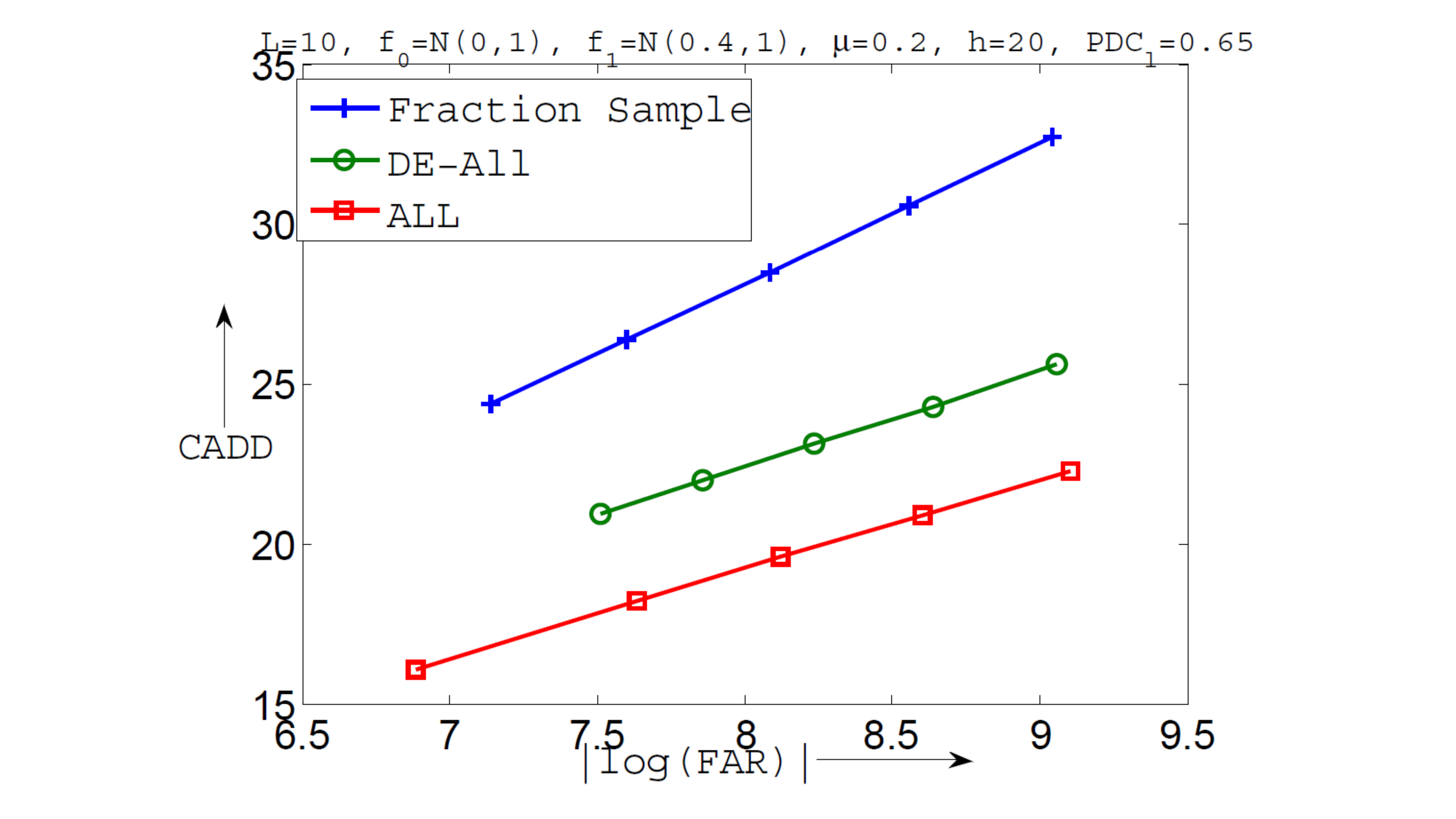}
\caption{{\footnotesize Trade-off curves for the algorithms studied: $L=10$,
$f_0 = f_{0,\ell}=\mathcal{N}(0,1)$, $\forall \ell$, and $f_1=f_{1,\ell}=\mathcal{N}(0.4,1)$, $\forall \ell$.
The values of $\mu = \mu_\ell = 0.2$, and $h=h_\ell=20$ are used to satisfy a $\PDC_\ell$ constraint of $0.65$ for each $\ell$. .}}
\label{fig:CompPlot}
\end{figure}

\section{Proof of Theorem~\ref{thm:DEALL}}
\label{sec:DEALL_Proof}
We first define some quantities and set the notation to be used in the proof of Theorem~\ref{thm:DEALL}.
Let $\tauwl(x,y)$ be the time taken for the DE-CuSum statistic at sensor $\ell$ to reach $y$,
starting at $W_{0,\ell}=x$.
Formally, for $x<y$ let
\[\tauwl(x,y) = \inf\{n \geq 1: W_{n,\ell} > y;  W_{0,\ell}=x\}.\]
If $x \geq y$, then $\tauwl(x,y)=0$.
Similarly, define
\[\taucl(x,y) = \inf\{n \geq 1: C_{n,\ell} > y;  C_{0,\ell}=x\},\]
where $C_{n,\ell}$ is the CuSum statistic at sensor $\ell$ when a CuSum algorithm is employed
at sensor $\ell$.
Also define the corresponding time for a random walk to move from $x$ to $y$:
\[\taurl(x,y) = \inf\{n \geq 1: x+\sum_{k=1}^n \log \frac{f_{1,\ell}(X_{k,\ell})}{f_{0,\ell}(X_{k,\ell})} > y\}.\]

Let $\nuwl(y)$ be the last time below
$y$ for the DE-CuSum statistic,
i.e., for $y \geq 0$,
\[\nuwl(y) = \sup\{n \geq 1: W_{n,\ell} \leq y;  W_{0,\ell}=y\}.\]
Similarly define, the last exit times for the CuSum algorithm
\[\nucl(y) = \sup\{n \geq 1: C_{n,\ell} \leq y;  C_{0,\ell}=y\},\]
and for the random walk
\[\nurl = \sup\{n \geq 1: \sum_{k=1}^n \log \frac{f_{1,\ell}(X_{k,\ell})}{f_{0,\ell}(X_{k,\ell})} \leq 0\}.\]

For simplicity we refer to the stopping for the DE-All algorithm simply by $\tau_a$.

\begin{proof}[Proof of Theorem~\ref{thm:DEALL}]
Our proof follows the outline of the proof of Theorem 3 in \cite{mei-ieeetit-2005}, but the
details here are slightly more involved.

We obtain an upper bound on $\Expect_\gamma \left[ (\tau_a-\gamma)^+ | \boldsymbol{\mathcal{I}}_{\gamma-1} \right]$
that is not a function of $\gamma$ and the conditioning $\boldsymbol{\mathcal{I}}_{\gamma-1}$, and that scales as the lower bound in Theorem~\ref{thm:Sensor_LB}. The theorem is then established if we then take the essential supremum and then
the supremum over $\gamma$.

Let $\boldsymbol{\mathcal{I}}_{\gamma-1}=\boldsymbol{i}_{\gamma-1}$ be such that
$W_{\gamma-1, \ell}=w_\ell$, $w_\ell \in [-h_\ell, \infty)$, $\forall \ell$.
We first note that
\begin{equation}
\begin{split}
\Expect_\gamma  \left[ (\tau_a-\gamma)^+ |  \boldsymbol{\mathcal{I}}_{\gamma-1}=\boldsymbol{i}_{\gamma-1} \right]
\leq \Expect_1\left[\max_{1\leq \ell \leq L} \{ \tauwl(w_\ell, d_\ell A) + \nuwl (W_{\tauwl(w_\ell, d_\ell A)})\}\right].
\end{split}
\end{equation}
By definition $W_{\tauwl(w_\ell, d_\ell A)} \geq d_\ell A$.
See Fig.~\ref{fig:DECuSUmPassagetimePlot1} for a typical evolution of the DE-CuSum statistic
at a sensor $\ell$, showing the first passage time $\tauwl(w_\ell, d_\ell D)$,
and the last exit time $\nuwl(y_\ell)$, where $y_\ell:=W_{\tauwl(w_\ell, d_\ell A)}$ for simplicity of
representation.
\begin{figure}[htb]
\center
\includegraphics[width=9cm, height=5cm]{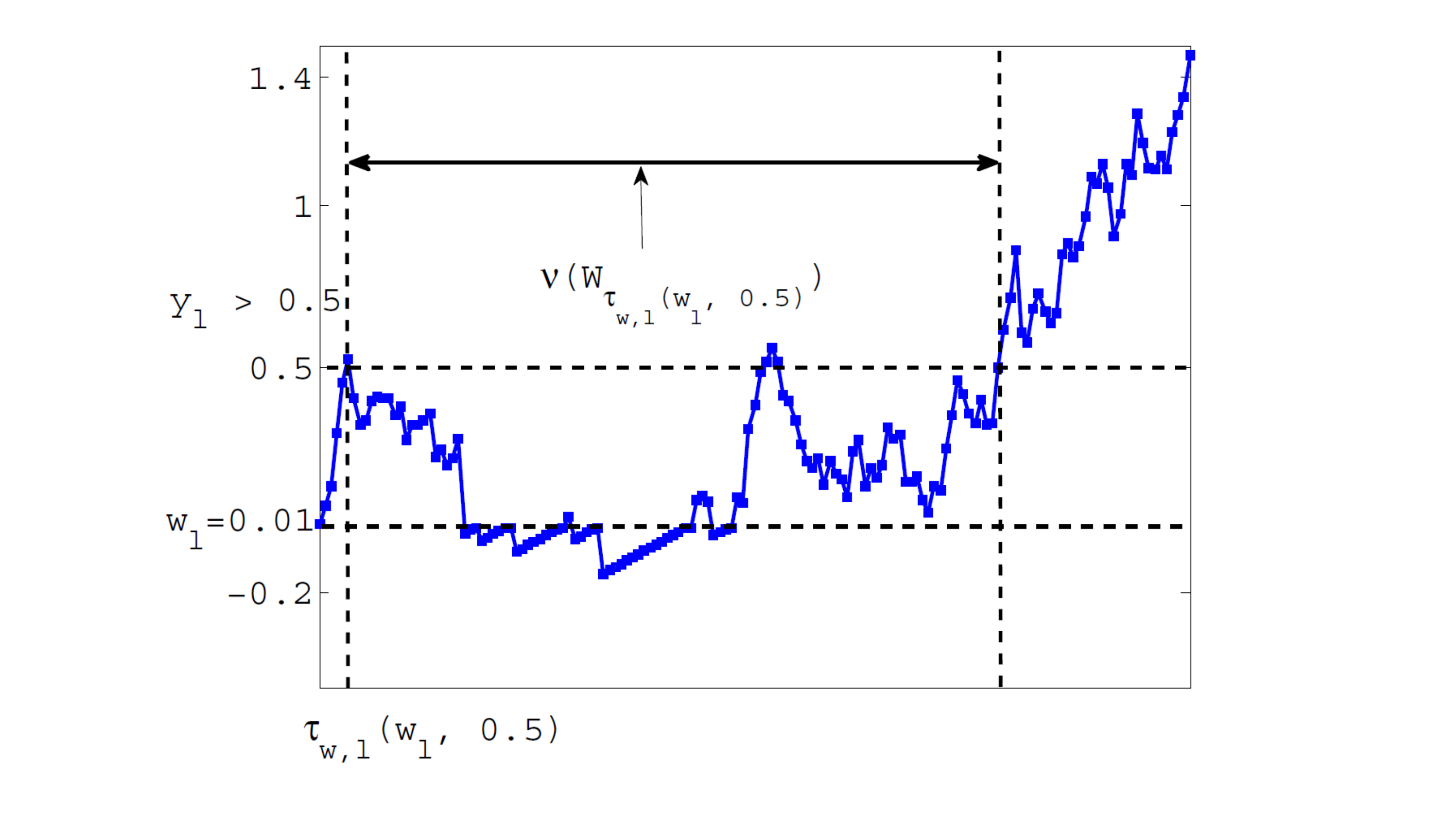}
\caption{Typical evolution of the DE-CuSum algorithm showing the first passage
time $\tauwl$, and the last exit time $\nuwl$ with $w_\ell=0.01$, $d_\ell A=0.5$}
\label{fig:DECuSUmPassagetimePlot1}
\end{figure}

It is easy to see that
\begin{equation}\label{eq:Appdeq1}
\begin{split}
\Expect_\gamma  \left[ (\tau_a-\gamma)^+ |  \boldsymbol{\mathcal{I}}_{\gamma-1}=\boldsymbol{i}_{\gamma-1} \right]
\leq  \; \Expect_1\left[\max_{1\leq \ell \leq L} \{ \tauwl(w_\ell, d_\ell A) \}\right]
+ \Expect_1\left[\max_{1\leq \ell \leq L} \{ \nuwl (W_{\tauwl(w_\ell, d_\ell A)})\}\right].
\end{split}
\end{equation}
We now show that the second term on the right hand side of \eqref{eq:Appdeq1}
is bounded by a constant, and the first term on the right hand side of \eqref{eq:Appdeq1}
is  $\frac{A}{\sum_{\ell=1}^L D(f_{1,\ell} \; || \; f_{0,\ell})} + O(\sqrt{A})$.

For the second term, from Lemma~\ref{lem:lem1} below, we have
\begin{equation}
\begin{split}
\Expect_1\left[\max_{1\leq \ell \leq L} \{ \nuwl (W_{\tauwl(w_\ell, d_\ell A)})\}\right]\leq
\sum_{\ell=1}^L \Expect_1\left[ \nuwl (W_{\tauwl(w_\ell, d_\ell A)})\right] \leq L K_3,
\end{split}
\end{equation}
where $K_3$ is a constant, not a function of the conditioning $w_\ell$, $d_\ell$, $\forall \ell$, and the threshold $A$.
Thus \eqref{eq:Appdeq1} can be written as
\begin{equation}\label{eq:Appdeq2}
\begin{split}
\Expect_\gamma \left[ (\tau_a-\gamma)^+ |  \boldsymbol{\mathcal{I}}_{\gamma-1}=\boldsymbol{i}_{\gamma-1} \right]
\leq  \; \Expect_1\left[\max_{1\leq \ell \leq L} \{ \tauwl(w_\ell, d_\ell A) \}\right]  + L K_3.
\end{split}
\end{equation}
For the first term on the right, we write the random variable $\tauwl(w_\ell, d_\ell A)$ in terms of
$\taucl(w_\ell, d_\ell A)$. We first assume that $0 \leq w_\ell \leq d_\ell A$.
Note that $\tauwl(w_\ell, d_\ell A)$ is the time for the DE-CuSum statistic $W_{n,\ell}$ to reach $d_\ell A$
starting with $W_{0,\ell} = w_\ell$.
And this time to hit $d_\ell A$ may have multiple sojourns of the statistic $W_{n,\ell}$ below 0.
Thus, the time $\tauwl(w_\ell, d_\ell A)$ can be written as the sum of random times.
Motivated by this we define a set of
new variables. In the following, we often suppress the dependence on the index $\ell$ for simplicity.

Let
\[
\tau_1(w_\ell) = \inf\{n\geq 1: W_{n,\ell} \not \in [0,d_\ell A] \mbox{ with } W_{0,\ell}=w_\ell\}.
\]
This is the first time for the DE-CuSum statistic, starting at $W_{0,\ell}=w_\ell$, to either hit $d_\ell A$
or go below $0$.
On paths over which $W_{\tau_1,\ell}<0$, we know that a number of consecutive samples are skipped
depending on the undershoot of the observations. Let $t_1(w_\ell)$ be the number
of consecutive samples skipped after $\tau_1(w_\ell)$ on
such paths.
On such paths again, let
\[
\tau_2(w_\ell) = \inf\{n> \tau_1(w_\ell)+t_1(w_\ell): W_{n,\ell} \not \in [0,d_\ell A]\}.
\]
Thus, on paths such that $W_{\tau_1,\ell}<0$,
after the times $\tau_1(w_\ell)$ and the number of skipped samples $t_1(w_\ell)$, the statistic $W_{n,\ell}$
reaches $0$
from below. The time $\tau_2(w_\ell)$ is the first time for
$W_{n,\ell}$ to either cross $A$ or go below $0$, after time $\tau_1(w_\ell)+t_1(w_\ell)$.
We define, $t_2(w_\ell)$, $\tau_3(w_\ell)$, etc. similarly. Next let,
\[
N_\ell(w_\ell) = \inf\{k \geq 1: W_{\tau_k,\ell} > d_\ell A\}.
\]

For simplicity we introduce the notion of ``cycles'', ``success'' and ``failure''. With reference to the
definitions of $\tau_k(w_\ell)$'s above, we say that a success has occurred if the statistic $W_{n,\ell}$, starting
with $W_{0,\ell}=w_\ell$, crosses $d_\ell A$ before it goes below $0$.
In that case we also say that the number of cycles to $d_\ell A$ is 1. If on the other hand,
the statistic $W_{n,\ell}$ goes below $0$ before it crosses $d_\ell A$,
we say a failure has occurred. On paths such that $W_{\tau_1,\ell}<0$,
and after the times $\tau_1(w_\ell)$ and the number of skipped samples $t_1(w_\ell)$, the statistic $W_{n,\ell}$
reaches $0$ from below. We say that the number of cycles is 2, if now
the statistic $W_{n,\ell}$ crosses $d_\ell A$ before it goes below $0$.
Thus,
$N_\ell(w_\ell)$ is the number of cycles to success at sensor $\ell$.

Let
\[
q_\ell = \Prob_1\left( \sum_{k=1}^n \log \frac{f_{1,\ell}(X_{k,\ell})}{f_{0,\ell}(X_{k,\ell})} \geq 0,\; \forall n\right).
\]
From \cite{wood-nonlin-ren-th-book-1982} it is well known that $q_\ell > 0$.
We claim that
\begin{equation}\label{eq:N_ell_UB}
\Expect_1[N_\ell(w_\ell)] \leq \frac{1}{q_\ell}.
\end{equation}
Thus, $N_\ell(w_\ell) < \infty$, a.s. $\Prob_1$.

If \eqref{eq:N_ell_UB} is indeed true then we can define
$\lambda_1(w_\ell) = \tau_1(w_\ell)$, $\lambda_2(w_\ell)= \tau_2(w_\ell) - \tau_1(w_\ell) - t_1(w_\ell)$, etc,
to be the lengths of the sojourns of the statistic $W_{n,\ell}$ above $0$.
Then clearly we have
\[
\tauwl(w_\ell, d_\ell A) = \sum_{k=1}^{N_\ell(w_\ell)}\lambda_k(w_\ell) +  \sum_{k=1}^{N_\ell(w_\ell)-1} t_k(w_\ell),
\]

If $w_\ell < 0$, then note that there will be an additional initial sojourn of the statistic
$W_{n,\ell}$ below $0$, equal to $\lceil |w_\ell^{h_\ell+}|/\mu_\ell \rceil$.
This is followed by delay term which corresponds to $w_\ell=0$.
Thus, in this case we can write
\[
\tauwl(w_\ell, d_\ell A) = \sum_{k=1}^{N_\ell(w_\ell)}\lambda_k(w_\ell) +  \sum_{k=1}^{N_\ell(w_\ell)} t_k(w_\ell),
\]
Such a statement is also valid even if $w_\ell > A$ because the right hand side of the above equation is positive.

Substituting this in \eqref{eq:Appdeq2} we have
\begin{equation}\label{eq:Appdeq3}
\begin{split}
\Expect_\gamma & \left[ (\tau_a-\gamma)^+ |  \boldsymbol{\mathcal{I}}_{\gamma-1}=\boldsymbol{i}_{\gamma-1} \right] \\
&\leq  \; \Expect_1\left[\max_{1\leq \ell \leq L} \left\{ \tauwl(w_\ell, d_\ell A) \right\}\right]  + L K_3 \\
&\leq  \; \Expect_1\left[\max_{1\leq \ell \leq L} \left\{ \sum_{k=1}^{N_\ell(w_\ell)}\lambda_k(w_\ell) +  \sum_{k=1}^{N_\ell(w_\ell)} t_k(w_\ell) \right\}\right]  + L K_3 \\
&\leq  \; \Expect_1\left[\max_{1\leq \ell \leq L} \left\{ \sum_{k=1}^{N_\ell(w_\ell)}\lambda_k(w_\ell) \right\} \right] +  \Expect_1\left[\max_{1\leq \ell \leq L} \left\{ \sum_{k=1}^{N_\ell(w_\ell)} t_k(w_\ell) \right\}\right]  + L K_3 \\
&\leq  \; \Expect_1\left[\max_{1\leq \ell \leq L} \left\{ \sum_{k=1}^{N_\ell(w_\ell)}\lambda_k(w_\ell) \right\} \right] +
\sum_{\ell=1}^L  \frac{\lceil h_\ell/\mu_\ell \rceil }{q_\ell}  + L K_3.
\end{split}
\end{equation}
The last inequality is true because
\[
t_k(w_\ell) \leq \lceil h_\ell/\mu_\ell \rceil, \; \forall w_\ell, k, \ell
\]
and because of \eqref{eq:N_ell_UB}.

We now make an important observation. We observe that because of the i.i.d. nature of the observations
\[
\taucl(w_\ell, d_\ell A) \stackrel{d}{=} \sum_{k=1}^{N_\ell(w_\ell)} \lambda_k(w_\ell),
\]
where we have used the symbol $\stackrel{d}{=}$ to denote equality in distribution.
Thus,
\[
\Expect_1\left[\max_{1\leq \ell \leq L} \left\{ \sum_{k=1}^{N_\ell(w_\ell)}\lambda_k(w_\ell) \right\} \right] =
\Expect_1\left[\max_{1\leq \ell \leq L} \taucl(w_\ell, d_\ell A) \right] .
\]
But, by sample-pathwise arguments it follows that
\[
\Expect_1\left[\max_{1\leq \ell \leq L} \taucl(w_\ell, d_\ell A) \right]  \leq
\Expect_1\left[\max_{1\leq \ell \leq L} \taucl(0, d_\ell A) \right].
\]
This gives us
\begin{equation}\label{eq:Appdeq4}
\begin{split}
\Expect_\gamma & \left[ (\tau_a-\gamma)^+ |  \boldsymbol{\mathcal{I}}_{\gamma-1}=\boldsymbol{i}_{\gamma-1} \right] \\
&\leq  \; \Expect_1\left[\max_{1\leq \ell \leq L} \left\{ \sum_{k=1}^{N_\ell(w_\ell)}\lambda_k(w_\ell) \right\} \right] +
\sum_{\ell=1}^L  \frac{\lceil h_\ell/\mu_\ell \rceil }{q_\ell}  + L K_3 \\
&=  \; \Expect_1\left[\max_{1\leq \ell \leq L} \taucl(w_\ell, d_\ell A) \right] +
\sum_{\ell=1}^L  \frac{\lceil h_\ell/\mu_\ell \rceil }{q_\ell}  + L K_3 \\
&\leq  \; \Expect_1\left[\max_{1\leq \ell \leq L} \taucl(0, d_\ell A) \right] +
\sum_{\ell=1}^L  \frac{\lceil h_\ell/\mu_\ell \rceil }{q_\ell}  + L K_3 \\
&\leq  \; \Expect_1\left[\max_{1\leq \ell \leq L} \taurl(0, d_\ell A) \right] +
\sum_{\ell=1}^L  \frac{\lceil h_\ell/\mu_\ell \rceil }{q_\ell}  + L K_3 \\
\end{split}
\end{equation}
We note that the right hand side of the above equation is not a function of $\gamma$ and the conditioning
$ \boldsymbol{\mathcal{I}}_{\gamma-1}=\boldsymbol{i}_{\gamma-1}$ anymore.
The theorem thus follows by taking $\esssup$ on the left hand side followed by a $\sup$ over time index
$\gamma$, and recalling the result from the proof of Theorem 3 of \cite{mei-ieeetit-2005} that $\Expect_1\left[\max_{1\leq \ell \leq L} \taurl(0, d_\ell A) \right]$
is of the order of $\frac{A}{\sum_{\ell=1}^L D(f_{1,\ell} \; || \; f_{0,\ell})} + O(\sqrt{A})$.

The proof of the theorem will be complete if we prove the claim \eqref{eq:N_ell_UB}.

With the identity
\[
\Expect_1[N_\ell(w_\ell)] = \sum_{k=1}^\infty \Prob_1(N_\ell(w_\ell) \geq k)
\]
in mind, and using the terminology of cycles, success and failure defined earlier, we write
\begin{equation*}
\begin{split}
\Prob_1(N_\ell(w_\ell) \geq k) &= \Prob_1(\mbox{fail in 1st cycle}) \; \Prob_1(\mbox{fail in $2^{nd}$ cycle}|\mbox{fail in first cycle})\\
                   &\; \; \cdots \Prob_1(\mbox{fail in $k-1^{st}$ cycle}|\mbox{fail in all previous}).
\end{split}
\end{equation*}
Now,
\begin{equation*}
\begin{split}
\Prob_1(\mbox{fail in $i^{th}$ cycle}|&\mbox{fail in all previous}) \\
&= 1 - \Prob_1(\mbox{success in $i^{th}$ cycle}|\mbox{fail in all previous}).
\end{split}
\end{equation*}

We note that
\begin{equation}\label{eq:ProbSuccLB}
\begin{split}
\Prob_1 &(\mbox{success in $1^{st}$ cycle}) = \Prob_1(W_{\tau_1,\ell}>A)\\
&=\Prob_1(\mbox{Statistic $W_{n,\ell}$ starting with $W_{0,\ell}=w_\ell$ reaches $d_\ell A$ before it goes below $0$})\\
&\geq \Prob_1\left( \sum_{k=1}^n \log \frac{f_{1,\ell}(X_{k,\ell})}{f_{0,\ell}(X_{k,\ell})} \geq 0,\; \forall n\right)=q_\ell.
\end{split}
\end{equation}
Here, the last inequality follows
because $\sum_{k=1}^n \log \frac{f_{1,\ell}(X_{k,\ell})}{f_{0,\ell}(X_{k,\ell})} \to \infty$ a.s. under $\Prob_1$, and hence
the statistic $W_{n,\ell}$ reaches $d_\ell A$ before actually never coming below $w_\ell$, and hence reaches $d_\ell A$ before going below $0$. Note that
the lower bound is not a function of the starting point $w_\ell$.

Similarly, for the second cycle
\begin{equation*}
\begin{split}
\Prob_1 &(\mbox{success in $2^{nd}$ cycle} | \mbox{failure in first})
 = \Prob_1 \left(W_{\tau_2,\ell}>A|W_{\tau_1,\ell}<0) \right)\\
&=\Prob_1\left(\mbox{Statistic $W_{n,\ell}$, for $n> \tau_1(w_\ell) + t_1(w_\ell)$, reaches $d_\ell A$ before it goes below $0$}\right)\\
&=\Prob_1(\mbox{Statistic $W_{n,\ell}$ starting with $W_{0,\ell}=0$ reaches $d_\ell A$ before it goes below $0$})\\
&\geq \Prob_1\left( \sum_{k=1}^n \log \frac{f_{1,\ell}(X_{k,\ell})}{f_{0,\ell}(X_{k,\ell})} \geq 0,\; \forall n\right)=q_\ell.
\end{split}
\end{equation*}
Almost identical arguments for the other cycles proves that
\[
\Prob_1(\mbox{success in $i^{th}$ cycle}|\mbox{fail in all previous}) \geq q_\ell, \; \forall i.
\]
As a result we get
\[
\Prob_1(N_\ell(w_\ell) \geq k) \leq (1-q_\ell)^{k-1}.
\]
Note that the right hand side is not a function of the initial point $w_\ell$, nor is a function of the threshold $A$.
Hence,
\begin{equation}
\Expect_1[N_\ell(w_\ell)] = \sum_{k=1}^\infty \Prob_1(N_\ell(w_\ell) \geq k) \leq \sum_{k=1}^\infty (1-q_\ell)^{k-1} = \frac{1}{q_\ell} < \infty.
\end{equation}
This proves the claim in \eqref{eq:N_ell_UB} and proves the theorem.
\end{proof}

\begin{lemma}\label{lem:lem1}
Let $\nuwl(w_\ell)$ as defined above be
the last exit time of the DE-CuSum statistic at sensor $\ell$ of the interval $(-\infty, w_\ell]$.
Then if the variance of the log likelihood ratio at sensor $\ell$ is finite,
then for all $w_\ell \geq 0$, and every $\ell$,
\begin{equation}
\begin{split}
\Expect_1 [\nuwl(w_\ell)] & \leq \Expect_1[\nucl(w_\ell)] + K_1 \\
& \leq   \Expect_1[\nurl(w_\ell)] + K_1 = K_3 < \infty,
\end{split}
\end{equation}
where $K_1$ and $K_3$ are finite positive constants.
\end{lemma}
\begin{proof}
Throughout the proof, we often suppress the dependence on the sensor index $\ell$.
The evolution of the DE-CuSum statistic from $n=1$ till $\nuwl(w_\ell)$
can be described as follows. The DE-CuSum starts at $w_\ell$, and initially
evolves like the CuSum algorithm, till either it goes below $0$,
or grows to $\infty$ without ever coming back to $0$.
Let $\mathcal{A}_1$ be the set of paths such that the DE-CuSum statistic grows to infinity without ever touching $0$.
In Fig.~\ref{fig:DECuSUmPassagetimePlot1}
consider the evolution of the DE-CuSum statistic by considering the time $\tauwl(w_\ell, d_\ell A)$
as the origin or time $n=0$. Then the sample shown in Fig.~\ref{fig:DECuSUmPassagetimePlot1} is a path from the set $\mathcal{A}_1^{c}$, which is the
complement of the set $\mathcal{A}_1$.
We define
\begin{equation}\label{eq:Lem1_eq1}
\begin{split}
\nu_1(w_\ell) &= \sup\{n \geq 1: W_{n,\ell} \leq w_\ell; W_{0,\ell}=w_\ell\} \;\;\; \text{ on } \; \mathcal{A}_1\\
          &= \inf\{n \geq 1: W_{n,\ell} < 0; \} \;\;\; \text{ on } \; \mathcal{A}_1^{c}
\end{split}
\end{equation}
Thus, on the set $\mathcal{A}_1$, $\nu_1(w_\ell)$ is the last exit time for the level $w_\ell$, and on the set $\mathcal{A}_1^{c}$,
$\nu_1(w_\ell)$ is the first time to hit $0$. We note that $\nu_1(w_\ell)$ is not a stopping time.

On the set $\mathcal{A}_1^{c}$, the DE-CuSum statistic goes below 0. Let $t_1(w_\ell)$ be the
time taken for the DE-CuSum statistic to grow up to 0, once it goes below 0 at $\nu_1(w_\ell)$.
Beyond $\nu_1(w_\ell) + t_1(w_\ell)$, the evolution of the DE-CuSum statistic is similar.
Either it grows up to $\infty$ (say on set of paths $\mathcal{A}_2$), or it goes below 0 (say on set of paths $\mathcal{A}_2^{c}$).
Thus, we define the variable
\begin{equation}
\begin{split}
\nu_2(w_\ell) &= \sup\{n > \nu_1(w_\ell) + t_1(w_\ell): W_{n,\ell} \leq w_\ell\} \;\;\; \text{ on } \; \mathcal{A}_2\\
          &= \inf\{n > \nu_1(w_\ell) + t_1(w_\ell): W_{n,\ell} < 0; \} \;\;\; \text{ on } \; \mathcal{A}_2^{c}
\end{split}
\end{equation}
The variables $t_3(w_\ell)$ and $\nu_3(w_\ell)$, etc., can be similarly defined. We note that the variables
here are similar to that used in the proof of Theorem~\ref{thm:DEALL}, but the variables $\nu_k(w_\ell)$s here are not stopping times.

Also, let
\[
N^{\nu}_\ell(w_\ell) = \inf\{ k\geq 1: W_{\nu_k(w_\ell),\ell} \geq 0\}.
\]

As done in the proof of Theorem~\ref{thm:DEALL}, we define the notion of ``cycles'', ``success'' and ``failure''.
With reference to the
definitions of $\nu_k(w_\ell)$'s above, we say that a success has occurred if the statistic $W_{n,\ell}$, starting
with $W_{0,\ell}=w_\ell$, grows to infinity before it goes below $0$.
In that case we also say that the number of cycles to the last exit time is 1. If on the other hand,
the statistic $W_{n,\ell}$ goes below $0$,
we say a failure has occurred. On paths such that $W_{\tau_1,\ell}<0$,
and after the times $\tau_1(w_\ell)$ and the number of skipped samples $t_1(w_\ell)$, the statistic $W_{n,\ell}$
reaches $0$ from below. We say that the number of cycles is 2, if now
the statistic $W_{n,\ell}$ grows to infinity before it goes below $0$.
Thus,
$N_\ell(w_\ell)$ is the number of cycles to success at sensor $\ell$.
See Fig.~\ref{fig:DECuSUmPassagetimePlot1}, where in the figure $N^{\nu}_\ell=7$.

Let
\[
q_\ell = \Prob_1\left( \sum_{k=1}^n \log \frac{f_{1,\ell}(X_{k,\ell})}{f_{0,\ell}(X_{k,\ell})} \geq 0,\; \forall n\right).
\]
We now show that
\[
\Expect_1[N^{\nu}_\ell(w)] \leq \frac{1}{q_\ell} < \infty.
\]
The last strict inequality is true because $q_\ell > 0$; see \cite{wood-nonlin-ren-th-book-1982}.

With the identity
\[
\Expect_1[N^{\nu}_\ell(w_\ell)] = \sum_{k=1}^\infty \Prob_1(N^{\nu}_\ell(w_\ell) \geq k)
\]
in mind, and using the terminology of cycles, success and failure defined above (and which are different
from those used in the proof of Theorem~\ref{thm:DEALL}), we write
\begin{equation*}
\begin{split}
\Prob_1(N^{\nu}_\ell(w_\ell) \geq k) &= \Prob_1(\mbox{fail in $1^{st}$ cycle}) \; \Prob_1(\mbox{fail in $2^{nd}$ cycle}|\mbox{fail in $1^{st}$ cycle})\\
                   &\; \; \cdots \Prob_1(\mbox{fail in $k-1^{st}$ cycle}|\mbox{fail in all previous}).
\end{split}
\end{equation*}
Now,
\begin{equation*}
\begin{split}
\Prob_1(\mbox{fail in $i^{th}$ cycle}|&\mbox{fail in all previous}) \\
&= 1 - \Prob_1(\mbox{success in $i^{th}$ cycle}|\mbox{fail in all previous}).
\end{split}
\end{equation*}

We note that
\begin{equation}
\begin{split}
\Prob_1 &(\mbox{success in $1^{st}$ cycle}) = \Prob_1(W_{\nu_1,\ell}>0)\\
&=\Prob_1(\mbox{Statistic $W_{n,\ell}$ starting with $W_{0,\ell}=w_\ell$ grows to $\infty$ before it goes below $0$})\\
&\geq \Prob_1\left( \sum_{k=1}^n \log \frac{f_{1,\ell}(X_{k,\ell})}{f_{0,\ell}(X_{k,\ell})} \geq 0,\; \forall n\right)=q_\ell.
\end{split}
\end{equation}
Here, the last inequality follows
because $\sum_{k=1}^n \log \frac{f_{1,\ell}(X_{k,\ell})}{f_{0,\ell}(X_{k,\ell})} \to \infty$ a.s. under $\Prob_1$, and hence
the statistic $W_{n,\ell}$ grows to infinity before never coming below $w_\ell \geq 0$. Note that
the lower bound is not a function of the starting point $w_\ell$.

Similarly, for the second cycle
\begin{equation*}
\begin{split}
\Prob_1 &(\mbox{success in $2^{nd}$ cycle} | \mbox{failure in first})
 = \Prob_1 \left(W_{\nu_2,\ell}>0|W_{\nu_1,\ell}<0) \right)\\
&=\Prob_1\left(\mbox{Statistic $W_{n,\ell}$, for $n> \nu_1(w_\ell)+ t_1(w_\ell)$, grows to $\infty$ before it goes below $0$}\right)\\
&=\Prob_1(\mbox{Statistic $W_{n,\ell}$ starting with $W_{0,\ell}=0$ grows to $\infty$ before it goes below $0$})\\
&= \Prob_1\left( \sum_{k=1}^n \log \frac{f_{1,\ell}(X_{k,\ell})}{f_{0,\ell}(X_{k,\ell})} \geq 0,\; \forall n\right)=q_\ell.
\end{split}
\end{equation*}
Almost identical arguments for the other cycles proves that
\[
\Prob_1(\mbox{success in $i^{th}$ cycle}|\mbox{fail in all previous}) \geq q_\ell, \; \forall i.
\]
As a result we get
\[
\Prob_1(N^{\nu}_\ell(w_\ell) \geq k) \leq (1-q_\ell)^{k-1}.
\]
Note that the right hand side is not a function of the initial point $w_\ell$, nor is a function of the threshold $A$.
Hence,
\begin{equation}\label{eq:ENx_UB}
\Expect_1[N^{\nu}_\ell(w_\ell)] = \sum_{k=1}^\infty \Prob_1(N_\ell(w_\ell) \geq k) \leq \sum_{k=1}^\infty (1-q_\ell)^{k-1} = \frac{1}{q_\ell} < \infty.
\end{equation}

Thus, $N^{\nu}_\ell(w_\ell) < \infty$, a.s. under $\Prob_1$ and we can define the following random variables:
$\lambda^{\nu}_1(w_\ell) = \nu_1(w_\ell)$, $\lambda^{\nu}_2(w_\ell)= \nu_2(w_\ell) - \nu_1(w_\ell) - t_1(w_\ell)$, etc,
to be the lengths of the sojourns of the statistic $W_{n,\ell}$ above $0$.
Then, $\nuwl(w_\ell)$ can be written as
\begin{equation}
\begin{split}
\Expect_1 [\nuwl(w_\ell)] = \Expect_1 \left[\sum_{k=1}^{N^{\nu}_\ell(w)} \lambda^{\nu}_1(w_\ell)\right] + \Expect_1 \left[\sum_{k=1}^{N^{\nu}_\ell(w)-1} t_k(w)\right].
\end{split}
\end{equation}
We observe that because of the i.i.d. nature of the observations
\[\Expect_1 [\nucl(w_\ell)] = \Expect_1 \left[\sum_{k=1}^{N^{\nu}_\ell(w)} \lambda^{\nu}_1(w_\ell)\right].\]
As a result,
\begin{equation}
\begin{split}
\Expect_1 [\nuwl(w)] = \Expect_1 [\nucl(w_\ell)] + \Expect_1 \left[\sum_{k=1}^{N^{\nu}_\ell(w)-1} t_k(w)\right].
\end{split}
\end{equation}
Now, $t_k(w) \leq \lceil h_\ell/\mu_\ell \rceil$, for any $k$, $w_\ell$, and every $\ell$.
Thus, we have
\begin{equation}
\begin{split}
\Expect_1 [\nuwl(w)] &\leq \Expect_1 [\nucl(w_\ell)] + \Expect_1 [N^{\nu}_\ell(w)] \lceil h_\ell/\mu_\ell \rceil \\
                     &\leq \Expect_1 [\nucl(w_\ell)] + \frac{\lceil h_\ell/\mu_\ell \rceil }{q_\ell}.
\end{split}
\end{equation}
The first inequality of the lemma follows from by setting $K_1 = \frac{\lceil h_\ell/\mu_\ell \rceil }{q_\ell}$.
The rest of the lemma follows by noting that by definition of the CuSum algorithm
\[\Expect_1 [\nucl(w_\ell)] \leq \Expect_1 [\nurl(w_\ell)],\]
and the latter is finite, and not a function of $w_\ell$, provided the variance of the log likelihood ratio is finite;
see \cite{mei-ieeetit-2005}.
\end{proof}

\section{CONCLUSIONS}
In this paper we proposed the DE-All algorithm, a data-efficient algorithm for sensor networks.
We showed that the proposed algorithm is first-order asymptotically optimal for QCD formulations 
where there is an additional constraint on the cost of observations used before the change point at each sensor. 
The results imply that one can skip an arbitrary but fixed fraction of samples before the change point, 
transmit a binary digit occassionally from the sensors to the fusion center, 
and still perform asymptotically up to the same order as
the Centralized CuSum algorithm. We note that in the latter algorithm, all the samples are used 
at each sensor, and raw observations are transmitted from the sensors to the fusion center at each time slot. 

One can expect better performance if in place of binary digits, more information is transmitted from the sensors
to the fusion center. In this case, one can use better fusion techniques at the fusion center to get improved performance. 
This is indeed true; see \cite{bane-veer-ssp-2012} and \cite{bane-veer-icassp-2013} for some preliminary results.

\section*{ACKNOWLEDGEMENTS}

This research was supported in part by the National Science Foundation (NSF) under grants
CCF 08-30169, CCF 11-11342 and DMS 12-22498 and by the Defense Threat Reduction Agency (DTRA)
under subcontract 147755 at the University of Illinois, Urbana-Champaign from prime award HDTRA1-10-1-0086.






\begin{thebibliography}{99}

\setlength{\parskip}{-1.5ex plus0ex minus1.1ex}


\bibitem[Banerjee et al.(2011)]{bane-etal-ieeetwc-2011}
Banerjee, T.,  Sharma, V., Kavitha, V., and Jayaprakasam, A. (2011).
Generalized Analysis of a Distributed Energy Efficient Algorithm for Change Detection,
{\em IEEE Transactions on Wireless Communication} 10: 91--101.

\bibitem[Banerjee and Veeravalli (2012a)]{bane-veer-sqa-2012}
Banerjee, T.  and Veeravalli, V. V. (2012a).
Data-Efficient Quickest Change Detection with On-Off Observation Control,
{\em Sequential Analysis} 31: 40–77.

\bibitem[Banerjee and Veeravalli (2012b)]{bane-veer-ssp-2012}
Banerjee, T.  and Veeravalli, V. V. (2012b).
Energy-Efficient Quickest Change Detection in Sensor Networks,
{\em in Proceedings of IEEE Statistical Signal Processing Workshop (SSP)}, AUG 5--8, Ann Arbor, MI, USA.

\bibitem[Banerjee and Veeravalli (2013a)]{bane-veer-IT-2013}
Banerjee, T.  and Veeravalli, V. V. (2013a).
Data-Efficient Quickest Change Detection in Minimax Settings,
{\em IEEE Transactions on Information Theory} 59: 6917–6931.

\bibitem[Banerjee and Veeravalli (2013b)]{bane-veer-icassp-2013}
Banerjee, T.  and Veeravalli, V. V. (2013b).
Data-efficient Quickest Change Detection in Distributed and Multichannel Systems,
{\em in Proceedings of International Conference on Acoustics, Speech, and Signal Processing (ICASSP) }, MAY 26--31, Vancouver, Canada.

\bibitem[Banerjee et al. (2013c)]{bane-veer-tart-isit-2013} Banerjee, T.  and Veeravalli, V. V. and Tartakovksy, A. (2013c).
Decentralized Data-Efficient Quickest Change Detection,
{\em in Proceedings of IEEE International Symposium on Information Theory (ISIT)}, JUL 7--12, Istanbul, Turkey.

\bibitem[Banerjee and Veeravalli (2014)]{bane-veer-isit-2014} Banerjee, T.  and Veeravalli, V. V. (2014).
Data-Efficient Quickest Change Detection with Unknown Post-Change Distribution,
{\em in Proceedings of IEEE International Symposium on Information Theory (ISIT)}, JUN 29--JUL 4, Honolulu, USA.


\bibitem[Lai(1998)]{lai-ieeetit-1998} Lai, T. L. (1998).
Information Bounds and Quick Detection of Parameter Changes in Stochastic Systems,
{\em IEEE Transactions on Information Theory} 44: 2917-2929.

\bibitem[Lorden (1971)]{lord-amstat-1971} Lorden, G. (1971).
Procedures for Reacting to a Change in Distribution,
{\em Annals of Mathematical Statistics} 42: 1897–1908.

\bibitem[Mei (2005)]{mei-ieeetit-2005} Mei, Y. (2005).
Information Bounds and Quickest Change Detection in Decentralized Decision Systems,
{\em IEEE Transactions on Information Theory} 51: 2669--2681.

\bibitem[Mei (2010)]{mei-biometrica-2010} Mei, Y. (2010).
Efficient Scalable Schemes for Monitoring a Large Number of Data Streams,
{\em Biometrika} 97: 419--433.






\bibitem[Mei (2011)]{mei-isit-2011} Mei, Y. (2011).
Quickest Detection in Censoring Sensor networks,
{\em in Proceedings of IEEE International Symposium on Information Theory (ISIT)}, JUL 31--AUG 5, Saint-Petersburg, Russia.


\bibitem[Moustakides (1986)]{mous-astat-1986} Moustakides, G. V. (1986).
Optimal Stopping Times for Detecting Changes in Distributions,
{\em Annals of Statistics} 14: 1379–1387.

\bibitem[Page (1954)]{page-biometrica-1954} Page, E. S. (1954).
Optimal Ddetection of a Change in Distribution,
{\em Biometrika} 41: 100–115.

\bibitem[Pollak (1985)]{poll-astat-1985} Pollak, M. (1985).
Optimal Detection of a Change in Distribution,
{\em Annals of Statistics} 13: 206–227.

\bibitem[Poor and Hadjiliadis (2009)]{poor-hadj-qcd-book-2009} Poor, H. V. and Hadjiliadis, O. (2009).
{\em Quickest Detection}, Cambridge University Press.

\bibitem[Premkumar and Kumar (2008)]{prem-kuma-infocom-2008} Premkumar, K. and Kumar, A. (2008).
 Optimal Sleep/Wake Scheduling for Quickest Intrusion Detection using Sensor Networks,
 {\em in Proceedings of IEEE Conference on Computer Communications (INFOCOM)}, APR 15--17, Phoenix, AZ, USA.






\bibitem[Shiryaev (1963)]{shir-siamtpa-1963} Shiryaev, A. N. (1963).
On Optimal Methods in Quickest Detection Problems,
{\em Theory of Probability and Its Applications} 8: 22--46.

\bibitem[Siegmund (1985)]{sieg-seq-anal-book-1985} Siegmund, D. (1985).
{\em Sequential Analysis: Tests and Confidence Intervals}, New York: Springer.

\bibitem[Tartakovsky et al. (2014)]{tart-niki-bass-2014}
Tartakovsky, A. G.,  Nikiforov, I. V., and Basseville, M. (2014).
{\em Sequential Analysis: Hypothesis Testing and Change-Point Detection}, Statistics, CRC Press.

\bibitem[Tartakovsky and Veeravalli (2002)]{tart-veer-fusion-2002} Tartakovsky, A. G. and Veeravalli, V. V. (2002).
 An Efficient Sequential Procedure for Detecting Changes in Multichannel and Distributed systems,
 {\em in Proceedings of IEEE International Conference on Information Fusion (ICIF)}, JUL 8--11, Annapolis, MD, USA.

\bibitem[Tartakovsky and Veeravalli(2005)]{tart-veer-siamtpa-2005} Tartakovsky, A. G. and Veeravalli, V. V. (2005).
General Asymptotic Bayesian Theory of Quickest Change Detection,
{\em Theory of Probability and Its Applications} 49:458--497.








\bibitem[Tartakovsky and Veeravalli(2008)]{tart-veer-sqa-2008} Tartakovsky, A. G. and Veeravalli, V. V. (2008).
Asymptotically Optimal Quickest Change Detection in Distributed Sensor Systems,
{\em Sequential Analysis} 27: 441--475.

\bibitem[Veeravalli(2001)]{veer-ieeetit-2001} Veeravalli, V. V. (2001).
Decentralized Quickest Change Detection,
{\em IEEE Transactions on Information Theory} 47: 1657--1665.

\bibitem[Veeravalli and Banerjee (2013)]{veer-bane-elsevierbook-2013} Veeravalli, V. V. and Banerjee, T. (2013).
{\em Quickest Change Detection},
Elsevier: E-reference Signal Processing http://arxiv.org/abs/1210.5552.

\bibitem[Wald and Wolfowitz (1948)]{wald-wolf-amstat-1948} Wald, A. and Wolfowitz, J. (1948).
Optimum Character of the Sequential Probability Ratio Test,
{\em Annals of Mathematical Statistics} 19: 326–339.


\bibitem[Woodroofe(1982)]{wood-nonlin-ren-th-book-1982} Woodroofe, M. (1982).
{\em Nonlinear Renewal Theory in Sequential Analysis}, SIAM, Philadelphia.

\bibitem[Zacharias and Sundaresan(2007)]{zach-sund-ieeetwc-2008} Zacharias, L. and Sundaresan, R. (2007).
Decentralized Sequential Change Detection using Physical Layer Fusion,
{\em in Proceedings of IEEE International Symposium on Information Theory (ISIT)}, June 24--29, Nice, France.

\end{thebibliography}
\end{document}